\documentclass{amsart}

\usepackage{cite}
\usepackage[all]{xy}
\usepackage{graphicx}
\usepackage{textcomp}
\usepackage{charter}
\usepackage[vflt]{floatflt}
\usepackage{amssymb}
\usepackage{enumerate}
\usepackage{float}
\usepackage{mathdots}
\usepackage[pdftex,colorlinks]{hyperref}
\hypersetup{
pdfauthor={Shamovich, E.},
pdfkeywords={lie group,lie algebra,noncommutative operators,vessel},
bookmarksnumbered,
pdfstartview={FitH}}%

\newtheorem{thm}{Theorem}[section]
\newtheorem{prop}[thm]{Proposition}
\newtheorem{lem}[thm]{Lemma}
\newtheorem{cor}[thm]{Corollary}

\newtheorem*{thm*}{Theorem}

\theoremstyle{definition}

\newtheorem{ex}[thm]{Example}
\newtheorem{rem}[thm]{Remark}

\theoremstyle{remark}

\numberwithin{equation}{section}

\newcommand{\buc}{\bu^{\times}_D}
\newcommand{\bucl}{\bu^{\times}_{D,\ell}}

\newcommand{\ord}{\operatorname{ord}}

\newcommand{\nocontentsline}[3]{}
\newcommand{\tocless}[2]{\bgroup\let\addcontentsline=\nocontentsline#1{#2}\egroup}

\usepackage{mathrsfs}

\newcommand{\im}{\operatorname{im}}
\newcommand{\Hom}{\operatorname{Hom}} 

\newcommand{\Supp}{\operatorname{Supp}}

\newcommand{\rank}{\operatorname{rank}} 

\newcommand{\cL}{{\mathcal L}}

\newcommand{\cO}{{\mathcal O}}

\newcommand{\cR}{{\mathcal R}}

\newcommand{\C}{{\mathbb C}}
\newcommand{\R}{{\mathbb R}}
\newcommand{\pp}{\mathbb{P}}

\newcommand{\Z}{{\mathbb Z}}

\newcommand{\T}{{\mathbb T}}
\newcommand{\D}{{\mathbb D}}

\newcommand{\bu}{{\mathbf u}}

\newcommand{\bmu}{{\boldsymbol \mu}}
\newcommand{\bnu}{{\boldsymbol \nu}}

\begin{document}

\title{How to count zeroes of polynomials on quadrature domains using the Bezout matrix}

\author{Eli Shamovich}
\address{Dept. of Pure Mathematics, University of Waterloo, Waterloo, ON, Canada}
\email{eshamovi@uwaterloo.ca}

\author{Victor Vinnikov}
\address{Department of  Mathematics\\ 
Ben-Gurion University of the Negev\\
8410501 Beer-Sheva, Israel}
\email{vinnikov@math.bgu.ac.il}

\begin{abstract}
Classically, the Bezout matrix or simply Bezoutian of two polynomials is used to locate the roots of the polynomial and, in particular, test for stability. In this paper, we develop the theory of Bezoutians on real Riemann surfaces of dividing type. The main result connects the signature of the Bezoutian of two real meromorphic functions to the topological data of their quotient, which can be seen as the generalization of the classical Cauchy index. As an application, we propose a method to count the number of zeroes of a polynomial in a quadrature domain using the inertia of the Bezoutian. We provide examples of our method in the case of simply connected quadrature domains.
\end{abstract}

\maketitle
\tableofcontents

\section{Introduction} \label{sec:intro}

Locating zeroes of polynomials or more precisely verifying whether all of the zeroes of a given polynomial lie in a certain domain in $\C$ is an old problem indeed. In particular, a polynomial is called stable if all of its roots lie in a left half-plane or the unit disc. Questions of stability arise naturally when considering linear systems and applying Laplace transform. The two modes of stability just mentioned correspond to continuous-time and discrete-time linear systems, respectively. Roughly speaking stability of the polynomial in the denominator of the Laplace transform of the solution implies that the solution is stable under small perturbations, i.e, a small perturbation in the data results in small ripples in the solution through time. For comparison, an unstable solution will change drastically as time passes from the original when a small perturbation is introduced in the initial data. There are quite a few stability criteria available for use for the cases of the disc and the half-plane. The goal of this paper is to propose a method of testing the stability of a polynomial on a general quadrature domain. Our method is a generalization of the classical method that applies the Bezout matrix to get information on the number of zeroes of a polynomial with respect to a half-plane.

The Bezout matrix $B(f,g)$ or simply Bezoutian of two polynomials $f$ and $g$ is the matrix of coefficients of the polynomial
\[
\frac{f(t) g(s) - f(s) g(t)}{t - s}.
\]
The theory of Bezoutians originates in the works of Hermite, Sylvester, Cayley, and others. The excellent paper \cite{KreNai81} contains a detailed review of Bezoutians and related topics with detailed references to the original manuscripts. Bezoutians for matrix and operator valued polynomials were studied by Lerer, Rodman, and Tismenetsky (see for example \cite{LerTis82} and \cite{LRT84}). The main result of the paper is a generalization to the case of real Riemann surfaces of dividing type of the following classical result \cite[XV]{KreNai81}:

\begin{thm*}[Hermite]
Let $f$ be a polynomial and denote $\overline{f} = \overline{f(\bar{z})}$. Let $n_+$, $n_-$, and $n_0$, be the number of positive eigenvalues, negative eigenvalues and the dimension of the kernel of $-i B(f,\overline{f})$. Then $f$ has exactly $n_0$ roots in common with $\overline{f}$ and moreover, it has additional $n_+$ roots in the upper half-plane and additional $n_-$ roots in the lower half-plane. 
\end{thm*}

It is natural to ask what kind of domains in $\C$ admit a generalization of the above theorem? Two ingredients in the above construction stand out. First, the anti-holomorphic involution that in the case of the upper half-plane is just the conjugation. The complex conjugation turns the Riemann sphere into a compact real Riemann surface of dividing type (see Section \ref{sec:riemann_surfaces} for the definitions). Real Riemann surfaces of dividing type arose naturally in the works of Ahlfors \cite{Ahl47,Ahl50}, in the study of definite determinantal representations of algebraic curves \cite{HelVin07, Vin93, Vppf} and the study of non-selfadjoint operators and in particular dilation theory of semigroups \cite{LKMV, ShaVin17}. However, in this paper, we are interested primarily in another natural source of real Riemann surfaces of dividing type, namely, quadrature domains. Recall that a domain $U \subset \C$ is called a quadrature domain if there exists $a_1,\ldots,a_r \in U$, non-negative integers $m_1,\ldots,m_r$ and complex numbers $c_{jk}$ for $j=1,\ldots,r$ and $k = 0,\ldots,m_j$, such that for every absolutely integrable analytic function\footnote{There are definitions of quadrature domains with respect to other classes of functions, but we will use this definition throughout the paper} $f$ on $U$, we have:
\[
\int_{U} f d\mu = \sum_{j=1}^r \sum_{k=0}^{m_j} c_{jk} f^{(k)}(a_j).
\]
Here $\mu$ stands for the Lebesgue measure on the plane. The study of such domains was initiated by Davis \cite{Dav74} and independently by Aharonov and Shapiro \cite{AhaSha76}. The expository article \cite{GusSha05} contains a wealth of information on the history and applications of quadrature domains. It was observed by Gustafsson \cite{Gus83} that $U$ is a quadrature domain if and only if there exists a meromorphic function on $X$, the Schottky double of $U$ (a compact Riemann surface of dividing type obtained by ``gluing'' two copies of $U$ along the boundary) univalent on the ``upper half-plane'' of $X$ and that maps the upper half-plane conformally onto $U$. We discuss quadrature domains in detail in Section \ref{sec:quadrature}.

The second ingredient is the function $\frac{1}{t-s}$, that happens to be the Cauchy kernel on the Riemann sphere, with the involution given by the complex conjugation. Cauchy kernels on Riemann surfaces have been introduced by Ball and the second author in \cite{BV-ZPF} and studied in \cite{AlpVin01} and \cite{AlpVin02} as reproducing kernels for spaces of functions on a Riemann surface. With the two ingredients at hand, we can now define the Bezout matrix of two meromorphic functions on a Riemann surface. Bezoutians in genus $0$ were used by Kravitsky \cite{LKMV} to construct determinantal representations of plane projective curves of genus $0$. Bezoutians on Riemann surfaces of higher genus were already implicit in the works Ball and the second author \cite{BV-ZPF} and Helton and the second author \cite{HelVin07}. Bezoutians on Riemann surfaces were defined in terms of determinantal representations by Shapiro and the second author in \cite{ASV1} and \cite{ASV2}. The definition of a Bezoutian on a Riemann surface using the Cauchy kernel was first introduced in \cite{SV14}. In \cite{SV14} the authors have studied determinantal representations of space projective algebraic curves and in particular, definite determinantal representations. A related notion of resultant on a compact Riemann surface was introduced in \cite{GusTk09} and was used to study the exponential transform associated with quadrature domains (see also \cite{GusTk11}).

The main result of this paper is Theorem \ref{thm:count_quadrature} that allows one to count the number of zeroes of a given polynomial $p$ on a quadrature domain $U$. To do this we consider the extension of the coordinate function $z$ to a meromorphic function on the Schottky double $X$ of $U$ and denote it by $\varphi$. Then $p$ also extends to a meromorphic function on $X$ that we will denote by $p$ as well. Since $X$ is a real Riemann surface it comes with an anti-holomorphic involution $\tau$. Let $p^{\tau}(z) = \overline{p(z^{\tau})}$ and consider $B_{\chi}(p,p^{\tau})$, with respect to a choice of a flat line bundle $\chi$ and a line bundle of half-order differentials $\Delta$. Let $J$ be a signature matrix that associated to the divisor of poles of $f$ and $g$ with respect to $\chi$. Then we have the following theorem (Theorem \ref{thm:count_quadrature}):

\begin{thm*}
Let $n_+$, $n_-$, and $n_0$ stand for the number of positive eigenvalues, negative eigenvalues and the dimension of the kernel of $- i J B_{\chi}(p,p^{\tau})$, respectively. The number of zeroes that $p$ and $p^{\tau}$ have in common is precisely $n_0$. Furthermore, $p$ has additionally $n_- - \deg_X p$ zeroes in $U$, where $\deg_X p$ stands for the degree of $p$ as a meromorphic function on $X$ ($\deg_X p = \deg p \deg_X \varphi$).
\end{thm*}

\textbf{Structure of the paper} We start in Section \ref{sec:riemann_surfaces} by introducing the basic definitions and constructions to be used throughout the paper. In particular, we define compact Riemann surfaces of dividing type and the Cauchy kernel. In this section, we prove some basic properties of the Cauchy kernel and constructions associated with it.

Section \ref{sec:bezoutians} is the technical heart of the paper. In this section, we define and study Bezoutians $B_{\chi}(f,g)$ of two real meromorphic functions $f$ and $g$ on a compact Riemann surface $X$ of dividing type in full generality (compared to \cite{SV14} where due to the nature of the application, Bezoutians were used only for pairs of meromorphic functions with simple poles). We show that the Bezoutian is a $J$-Hermitian matrix, with respect to a matrix $J$ arising from the choice of a line bundle in the construction of the Cauchy kernel and the join of the pole divisors of the meromorphic functions $f$ and $g$. Then we proceed to study the sesquilinear form defined by $J B_{\chi}(f,g)$ and construct bases that give this bilinear form a particularly nice representation. The main theorem of this section (Theorem \ref{thm:signature}) determines the inertia and the signature of $J B_{\chi}(f,g)$ in terms of the topological data coming from the meromorphic function $h = f/g$,

Section \ref{sec:main_result} contains the main results of the paper. Theorem \ref{thm:from_sig_to_coh} states that the signature of $J B_{\chi}(f,g)$ is the Cauchy index of $h= f/g$, Theorem \ref{thm:argument} can be viewed as a generalization of the argument principle. It states that for every $\lambda \in \C$, such that $\im(\lambda) >0$, the signature of $J B_{\chi}(f,g)$ is the difference between the number of points in the fiber of $h$ over $\lambda$ that lie in the upper half-plane of $X$ and the number of points in this fiber that lie in the lower half-plane of $X$ (counting multiplicities). From these two theorems, we derive several corollaries, that will be used in applications.

The last section applies the theory developed in previous sections to the case of quadrature domains. We formulate and prove Theorem \ref{thm:count_quadrature} that allows us to count the number of zeroes of a given polynomial on a quadrature domain $U$ in terms of the inertia of the Bezoutian associated to $p$ with its conjugate on the Schottky double of $U$. Then we provide examples of concrete calculations in the genus $0$ case in subsection \ref{sec:genus_0}. Note that the question of stability of a polynomial even with respect to a bounded simply connected quadrature domain is non-trivial. The only exception is the unit disc, where one can use the classical Schur-Cohn theory. Therefore, even in this simple case, our main result is new.	

\textbf{Acknowledgments} The first author was partially supported by the Fields Institute for Research in the Mathematical Sciences. The first author also thanks Prof. Kenneth R. Davidson and the Department of Pure Mathematics at the University of Waterloo for their warm welcome and hospitality. The research of the second author was partially supported by the Deutsche Forschungsgemeinschaft (DFG) and the Israel Science Foundation (ISF).

\section{Riemann Surfaces and Cauchy Kernels} \label{sec:riemann_surfaces}

A Riemann surface $X$ is a complex manifold of complex dimension $1$. From now on we abbreviate 
``compact Riemann surface'' to simply ``Riemann surface''. The genus of a Riemann surface will be denoted by $g$, where the genus is the topological genus, i.e, the ``number of holes'' in $X$. In this paper, holomorphic line bundles on Riemann surfaces are a key tool. Hence we will now recall the basic notations and terminology of line bundles and  related topics. Recall that a line bundle on $X$ is a complex manifold $L$ together with a holomorphic projection $\pi \colon L \to X$, that is in addition locally trivial, i.e, for every $x \in X$, there exists an open neighborhood $x \in U \subset X$, such that $\pi^{-1}(U)$ is biholomorphic to $U \times \C$. In other words, a line bundle is a holomorphic family of one-dimensional complex vector spaces parametrized by $X$. We will consider the space of holomorphic sections of a line bundle $L$, namely the holomorphic functions $\sigma \colon X \to L$, such that $\pi \circ \sigma = \operatorname{id}_X$. We will denote the space of all holomorphic sections of $L$ by $H^0(X,L)$ or simply by $H^0(L)$, if $X$ is clear. Since $X$ is compact, it turns out that this space is finite dimensional and we will denote its dimension by $h^0(L)$. Several special examples of line bundles are of particular interest, The trivial line bundle is $X \times \C$ and the global sections are just global holomorphic functions on $X$, i.e, the constant functions. We will denote the trivial line bundle by $\cO$ and thus $H^0(X,\cO) = \C$ and $h^0(\cO) = 1$. The holomorphic cotangent bundles is another example of a holomorphic line bundle on $X$. We will denote it by $\Omega_X$. It is a theorem that $h^0(X,\Omega_X) = g$ or in other words, that there are $g$ linearly independent global holomorphic $1$-forms on $X$.

A morphism of line bundles from $(L_1,\pi_1)$ to $(L_2,\pi_2)$ is a holomorphic map $\varphi \colon L_1 \to L_2$, such that $\pi_2 \circ \varphi = \pi_1$ and on each fiber $\varphi$ induces a linear map. An isomorphism of line bundles is a morphism that has an inverse morphism. To each line bundle, we can associate the dual line bundle $L^{\vee}$, where each fiber is the dual space of the corresponding fiber of $L$. Thus, in particular, $\Omega_X^{\vee}$ is the holomorphic tangent bundle on $X$ and $\cO^{\vee} = \cO$. We can tensor two line bundles to obtain a new line bundle $L_1 \otimes L_2$. In particular, for every line bundle $L$, $L \otimes L^{\vee} \cong \cO$, and $L \otimes \cO \cong L$. This implies that the isomorphism classes of line bundles form a group with the operation being the tensor product and the unit element is the trivial bundle. This group is called the Picard group of $X$. The Serre dual of a line bundle $L$ is $L^{\vee} \otimes \Omega_X$. The celebrated Serre duality restricted to the case of compact Riemann surfaces is the statement that $H^1(X,L) = H^0(L^{\vee} \otimes \Omega_X)$. Here $H^1$ is the first sheaf cohomology group of the sheaf of sections of $L$, however, a reader unfamiliar with sheaf cohomology may take the above equality as the definition. In particular, we will denote $h^1(L) = h^0(L^{\vee} \otimes \Omega_X)$. For example $h^1(\cO) = g$ and $h^1(\Omega_X) = 1$. 

A divisor $D$ on a Riemann surface $X$ is an element of the free abelian group generated by the points of $X$, or in other words, a finite linear combination $D = \sum_{j=1}^m n_j p_j$, where $n_1,\ldots,n_m \in \Z$ and $p_1,\ldots,p_m \in X$. For example, if $f$ is a meromorphic function on $X$, then we can associate to $f$ the divisor $(f) = \sum_{j=1}^m \nu_{p_j}(f) p_j$, where each $p_j$ is either a zero or a pole of $f$ and $\nu_{p_j}(f)$ is the order of the zero/pole at $p_j$. The divisor $(f)$ is called a principal divisor. The principal divisors form a subgroup of the group of divisors on $X$. The degree of a divisor $D = \sum_{j=1}^m n_j p_j$ is $\deg D = \sum_{j=1}^m n_j$ and in particular, $\deg (f) = 0$. A divisor $D$ is called positive and denoted by $D \geq 0$, if each coefficient is non-negative. To each divisor, we can associate the line bundle $\cO(D)$, whose global sections are the meromorphic functions on $X$, such that $(f) + D \geq 0$. Two such line bundles are isomorphic if only if the divisors differ by a principal divisor. Hence we get a homomorphism from the group of divisors on $X$ modulo the principal divisor into the Picard group. It turns out that this map is an isomorphism since we can associate to each line bundle, a divisor up-to a principal divisor. A line bundle is called flat if the degree of the associated divisor is $0$. We refer the reader to the excellent books \cite{FarKra92,Gun66,Gun72,Mir95} for further details on these topics.

A real Riemann surface $X$ is a Riemann surface equipped with an anti-holomorphic involution $\tau$. We will denote by $X(\R)$ the fixed points of $\tau$ and call such points the real points of $X$. We say that $X$ is dividing if $X \setminus X(\R)$ has two connected components interchanged by $\tau$ or alternatively if $X/\tau$ is orientable. In this case, we fix one component and denote it $X_{+}$ and we fix an orientation on $X(\R)$, such that $X(\R)$ becomes the boundary of $X_{+}$. We will write $X_{-}$ for the other component. The anti-holomorphic involution acts on functions by $f^{\tau}(p) = \overline{f(p^{\tau})}$ and on line bundles on $X$ (more generally on sheaves on $X$). We say that a function is real if $f = f^{\tau}$.

Given a real Riemann surface $X$, in \cite{Vin93} the second author has constructed a canonical integral homology basis on $X$. The basis has the property that it is a symplectic basis with respect to the intersection pairing and the matrix representing $\tau$ with respect to this basis is of the form $\left( \begin{smallmatrix} I & 0 \\ H & -I \end{smallmatrix} \right)$. For details and a precise description of $H$, the reader is referred to \cite[Proposition 2.2]{Vin93} and the discussion following it. We will denote the basis by $A_1,\ldots,A_g,B_1,\ldots,B_g \in H_1(X,\Z)$. We fix a basis $\omega_1, \ldots,\omega_g \in H^0(X,\Omega_X)$ for the space of holomorphic differentials on $X$ normalized with respect to the canonical homology basis, in the sense that $\int_{A_i} \omega_j = \delta_{ij}$. With this data at hand, for a fixed point $x_0 \in X$, the Abel-Jacobi map is $\varphi(x) = \left( \begin{smallmatrix} \int_{x_0}^x \omega_1, \ldots, \int_{x_0}^x \omega_g \end{smallmatrix} \right)$. The Abel-Jacobi map maps $X$ into the Jacobian variety of $X$, $J(X) = \C^g/(\Z^g + \Lambda \Z^g)$, where $\Lambda$ is the $B$-period matrix of the basis of differentials chosen above. The Jacobian is a complex torus and thus, in particular, is an abelian group. Therefore, one extends the Abel-Jacobi map to the divisors on $X$, i.e., formal integral combinations of points on $X$. The involution $\tau$ extends linearly to divisors on $X$. Due to the choices that we have made the differentials $\omega_1,\ldots,\omega_g$ are real and the complex conjugation on $\C^g$ preserves the lattice generated by the columns of the period matrix and thus induces an anti-holomorphic involution on the Jacobian of $X$, that we shall also denote by $\tau$. Using this, we can deduce that for every divisor $D$, $\varphi(D^{\tau}) = \varphi(D)^{\tau} + \deg D \cdot \varphi(x_0^{\tau})$. In particular, if the base point is real, then $\varphi$ intertwines the involutions. i.e, it is a real map.

Recall that the Riemann theta function of $X$ is a function on $\C^g$:
\[
\theta(z) = \sum_{m \in \Z^g} \exp\left( 2\pi i \left( \frac{1}{2} m^T \Lambda m + m^T z\right) \right).
\]
Here $m^T$ stands for the transpose. Furthermore, one defines the Riemann theta function with characteristics $a,b \in \C^g$ as follows:
\[
\theta\left[ \begin{smallmatrix} a \\ b \end{smallmatrix}\right](z) = \exp\left(2 \pi i \left( \frac{1}{2} a^T \Lambda a + a (z +b)^T \right) \right) \theta(z + a \Lambda + b).
\]

Let us fix a line bundle of half-order differentials $\Delta$, such that $\varphi(\Delta) = - \kappa$, where $\kappa$ is the Riemann constant. That is $\Delta$ is a line bundle, such that $\Delta \otimes \Delta = \Omega_X$ and $\kappa \in J(X)$ is the vector defined by the equation
\[
\sum_{j=1}^g \int_{x_0}^{x_j} \omega_m = e_m - \kappa_m ( \operatorname{mod} \Z^g + \Lambda \Z^g).
\]
Here $x_1,\ldots,x_g$ are the zeroes of the Riemann theta function restricted to the image of $\varphi$. Let $\chi$ be a flat line bundle on $X$, such that $h^0(\chi \otimes \Delta) = h^1(\chi \otimes \Delta) = 0$. Let $\rho_1,\rho_2 \colon X\times X \to X$, be the projections on the first and second coordinate, respectively. In \cite{BV-ZPF} Ball and the second author have constructed the Cauchy kernel associated to $\chi$. Namely, $K_{\chi}$ is a global section of $\rho_1^*(\chi \otimes \Delta) \otimes \rho_2^*(\chi^{\vee} \otimes \Delta) \otimes \cO_{X \times X}(D_X)$, where $D_X$ is the divisor of the diagonal in the product. The uniqueness of $K_{\chi}$ stems from the fact that the residue along the diagonal map is an isomorphism from $H^0(\rho_1^*(\chi \otimes \Delta) \otimes \rho_2^*(\chi^{\vee} \otimes \Delta) \otimes \cO_{X \times X}(D_X))$ to $\Hom(\chi \otimes \Delta, \chi \otimes \Delta)$ (it takes $K_{\chi}$ to the identity map on $\chi \otimes \Delta$). One can write $K_{\chi}$ in terms of the theta function with characteristic and the prime form on $X$, as follows:
\[
K_{\chi}(p,q) = \frac{\theta \left[ \begin{smallmatrix} a \\ b \end{smallmatrix}\right]\left(\varphi(q) -\varphi(p)\right)}{\theta\left[ \begin{smallmatrix} a \\ b \end{smallmatrix}\right] \left(0\right) E_{\Delta}(q,p)}.
\]
Here $\varphi(\chi) = (a,b + \Lambda a)$ and $E_{\Delta}$ is the prime form associated with $\Delta$ (see \cite{Fay73}).

Let $\pi \colon \tilde{X} \to X$ be the universal cover of $X$ and denote by $\tilde{x} \in \tilde{X}$ a lift of $x \in X$. It is a well-known fact that the fundamental group of $X$ acts on $\tilde{X}$ and preserves the fibers. Choosing coordinates $t$ and $s$ around $p, q\in X$ and the positive branch of the square root on the coordinate patches, we have that:
\[
\frac{K_{\chi}(T \tilde{p},R \tilde{q})}{\sqrt{dt}(T \tilde{p}) \sqrt{ds}(R \tilde{q})} = a_{\chi}(T) \frac{K_{\chi}(\tilde{p},\tilde{q})}{\sqrt{dt}(\tilde{p}) \sqrt{ds}(\tilde{q})} a_{\chi}(R)^{-1}.
\]
Here $T,R \in \pi_1(X)$ and $a_{\chi}$ is the factor of automorphy of $\chi$.

Given an effective divisor $D$ on $X$, let $\cL = \cO_X(D) \otimes \chi \otimes \Delta$. By Riemann-Roch we know that $h^0(\cL) - h^1(\cL) = \deg D$. By Serre duality $h^1(\cL) = h^0(\cL^{\vee} \otimes \Omega_X) = h^0(\cO_X(-D) \times \chi^{\vee} \otimes \Delta).$ Note that since $D$ is effective and $h^0(\chi^{\vee} \otimes \Delta) = 0$ we conclude that $h^1(\cL) = 0$ and thus $h^0(\cL) = \deg D$. In fact, we can consrtuct a basis of $H^0(X,\cL)$ using $K_{\chi}$. To do this write $D = \sum_{j=1}^r n_j p_j$ and choose a local coordinate $t_j$ for every $p_j$. Choose a lift $\tilde{p}_j \in \tilde{X}$ for every $j=1,\ldots,r$. Now consider the meromorphic section of $\chi \otimes \Delta$ obtained by $\mu_{j,0}(p) = \frac{K_{\chi}(p, \tilde{p}_j)}{\sqrt{dt_j}(\tilde{p}_j)}$. As we have seen above changing the choice of the lift will only multiply this section by a non-zero scalar (since $\chi$ is flat). Choosing an open neighborhood $p_j \in U_j \subset X$, such that $\pi^{-1}(U_j)$ is homeomorphic to a disjoint union of copies of $U_j$, we can choose the lift $\tilde{p}_j$ consistently and consider the derivative of $\mu_{j,0}$ as a function of the second coordinate. Note that changing $\tilde{p}_j$ to $T \tilde{p}_j$ for some $T \in \pi_1(X)$, will multiply the derivative again by $a_{\chi}(T)$,  which is a non-zero scalar independent of $p_j$. Thus we can obtain $\mu_{j,k}(p)$ as the derivative of $\mu_{j,0}$ of order $k$. Since locally
\[
K_{\chi}(\tilde{p},\tilde{p}_j) = \frac{1}{t_j(\tilde{p}) - t_j(\tilde{p}_j)} + \text{ analytic terms},
\]
we see that $\mu_{j,k}$ is a section of $\chi \otimes \Delta$, with a unique pole at $p_j$ of order $k+1$. Thus the set $\left\{ \mu_{j,k} \mid j=1,\ldots,r,\, k=0,\ldots,n_j-1\right\}$ is a basis of $H^0(X,\cL)$. Similarly, if we write $\cR = \cO_X(D) \otimes \chi^{\vee} \otimes \Delta$ we obtain that $h^0(\cR) = \deg D$ and $h^1(\cR) = 0$. Specializing the Cauchy kernel at the first point, we can construct a basis for $H^0(X,\cR)$. We denote $\nu_{j,0}(p) = \frac{K_{\chi}(\tilde{p}_j,p)}{\sqrt{dt_j}(\tilde{p}_j)}$ and taking derivatives in the first coordinate we obtain the basis $\left\{\nu_{j,k} \mid j=1,\ldots,r,\, k =0,\ldots,n_j-1\right\}$.

Assume now that $X$ is a real Riemann surface of dividing type. Let $k$ be the number of connected components of $X(\R)$. The Jacobian of $X$ contains a set of real tori, parametrized by a choice of signs $v= (v_1,\ldots,v_{k-1}) \in \{0,1\}^{k-1}$ defined by:
\begin{multline*}
\T_v = \left\{\zeta \in J(X) \mid \zeta = \frac{v_1}{2} e_{g+ 2 - k} + \cdots + \frac{v_{k-1}}{2} e_g +  a_1 \left( \Lambda_1 - \frac{1}{2} e_2 \right) + a_2 \left( \Lambda_2 - \frac{1}{2} e_1 \right) \right.\\ + \left.  \cdots + a_{g- k} \left(\Lambda_{g-k} - \frac{1}{2} e_{g-k+1} \right) + a_{g+1 - k} \left( \Lambda_{g+1-k} - \frac{1}{2} e_{g-k} \right) \right.\\ + \left. a_{g+2 - k} \Lambda_{g+2-k} + \cdots + a_g \Lambda_g,\, (a_1,\ldots,a_g) \in \R^g\right\}.
\end{multline*}
Here $\Lambda_j$ are the columns of the period matrix $\Lambda$ and $e_1,\ldots,e_g$ is the standard basis of $\C^g$. If $\chi \in \T_v$, then $\chi \otimes \Delta$ has a special property that $\left(\chi \otimes \Delta \right)^{\tau} = \chi^{\vee} \otimes \Delta$. By \cite[Lemma 5.4]{SV14} for $\chi \in \T_v$ we know that for every two distinct $p,q \in X$, we have $\overline{K_{\chi}(p,q)} = K_{\chi}(q^{\tau}, p^{\tau})$. A divisor on $X$ is called real if $D = D^{\tau}$ and since $X$ is dividing, we can always write a real divisor $D = D_r + D_i + D_i^{\tau}$, where the support of $D_r$ is in $X(\R)$ and the support of $D_i$ is in $X_{+}$. Let $\cL =\cO_X(D) \otimes \chi \otimes \Delta$ again and recall that we have obtained a basis for $H^0(X,\cL)$ by using the Cauchy kernel. The basis depends on the choice of the lift $\tilde{p}_j$. We will be more specific choosing the lifts in the real case and for every $p_j$ in the support of $D_i$ we choose a lift $\tilde{p}_j$ and set $\tilde{p}^{\tau}_j$ for the lift of $p^{\tau}_j$.

\begin{lem} \label{lem:duality}
Let $X$ be of dividing type and $\chi \in \T_v$ and $D$ be a real divisor. Write $D = D_r + D_i + D_i^{\tau}$, with $D_r = \sum_{j=1}^{\ell} n_j p_j$ and $D_i = \sum_{j=\ell+1}^{r} n_j p_j$. Choose lifts $\tilde{p}_j$ as above. Let us write $\mu_{j^{\tau},k}$ and $\nu_{j^{\tau},k}$ for the $\mu$ and $\nu$ basis elements associated to $p_j^{\tau}$, for $j=\ell+1,\ldots,r$.

Then we have that for $j=0,\ldots,\ell$ and $k=0,\ldots,n_j-1$
\[
\overline{\nu_{j,k}} = - \sigma_{v,j} \mu_{j,k},
\]
here $\sigma_{v,j}$ is a sign that depends only on $v$ and the point $p_j$. Furthermore, for $j=\ell+1,\ldots,r$ nad $k = 0,\ldots,n_j-1$:
\[
\overline{\nu_{j,k}} = - \mu_{j^{\tau},k},\quad \overline{\nu_{j^{\tau},k}} = - \mu_{j,k}.
\]
\end{lem} 
\begin{proof}
By \cite[Corollary 5.5]{SV14} we know that the first equality holds for $k=0$. However, the fact that it holds for the derivatives is immediate, since the sign comes from the constant automorphy factor of $\chi$. The second equality is immediate from the properties of $K_{\chi}$ and our choice of lifts.
\end{proof}

We can reformulate the statement of the lemma above in a more compact way. Let $m = \deg D_i$ and set $J_v = I_v \oplus \left( \begin{smallmatrix} 0 & 1 \\ 1 & 0 \end{smallmatrix} \right)^{\oplus m}$, where $I_v$ is a diagonal matrix of size $\deg D_r$, with the signs $\sigma_{v,j}$ on the diagonal. We define the column and row vectors of sections:
\[
\buc = \begin{pmatrix} \nu_{1,0} & \cdots & \nu_{1,n_1-1} &\cdots & \nu_{\ell,n_{\ell}-1} & \nu_{\ell+1,0} & \nu_{(\ell+1)^{\tau},0} & \cdots &\nu_{r,n_r-1} & \nu_{r^{\tau},n_r-1} \end{pmatrix}^T.
\]
\[
\bucl = \begin{pmatrix} \mu_{1,0} & \cdots & \mu_{1,n_1-1} &\cdots & \mu_{\ell,n_{\ell}-1} & \mu_{\ell+1,0} & \mu_{(\ell+1)^{\tau},0} & \cdots &\mu_{r,n_r-1} & \mu_{r^{\tau},n_r-1} \end{pmatrix}
\]
Then we can write:
\begin{cor} \label{cor:u_times_star}
Let $X$ be  real Riemann surface of dividing type, $\chi \in \T_v$ and $D$ a real effective divisor on $X$, then:
\[
\bucl(p) = - \left(J_v \buc(p^{\tau}) \right)^*
\]
\end{cor}

\section{Bezoutians on Riemann Surfaces} \label{sec:bezoutians}

Let $X$ be a Riemann surface $D = \sum_{j=1}^m n_j p_j$  be an effective real divisor of degree $N = \sum_{j=1}^m n_j$ on $X$ and let $f, g \in \cL(D)$ be two meromorphic functions.  We consider the following expression for two distinct points $p,q \in X$:
\[
b_{\chi}(f,g)(p,q) = (f(p) g(q) - f(q) g(p) ) K_{\chi}(p,q).
\]

Then we have the following theorem that generalizes \cite[Prop.\ 4.1]{SV14}:

\begin{prop} \label{prop:bezoutian_formula}
There exists a matrix $B_{\chi}(f,g) \in M_N(\C)$, such that
\[
b_{\chi}(f,g) = \bnu(p) B_{\chi}(f,g) \bmu(p),
\]
where $\bnu$ and $\bmu$ are the row and column vectors consisting of the basis elements $\nu_{j,k}$ and $\mu_{j,k}$, respectively. This defines a linear mapping $\wedge^2 \cL(D) \to M_N(\C)$ given by $f \wedge g \mapsto B_{\chi}(f,g)$. 

Additionally, if $X$ is a real Riemann surface of dividing type, $D = D_r + D_i + D_i^{\tau}$ is a real divisor and $f$ and $g$ are real meromorphic functions, then if $\chi \in \T_v$, then we set $B_{\chi}(f,g) \in M_N(\C)$, such that:
\[
b_{\chi}(f,g)(p,q) = \bucl(p) B_{\chi}(f,g) \buc(q) = - \buc(p^{\tau})^* J_v B_{\chi}(f,g) \buc(q).
\]
The matrix $B_{\chi}(f,g)$ thus defined is $J_v$-Hermitian.
\end{prop}
\begin{proof}
Let $D_X \subset X \times X$ be the diagonal. Let $\pi_1$ and $\pi_2$ be the projections on the first and second coordinates, respectively. If $f$ and $g$ are viewed as sections of $\cO_X(D)$ and thus one can think of $b_{\chi}(f,g)$ as a section of $\pi_1^*\left( \cO_X(D) \otimes \chi \otimes \Delta\right) \otimes \pi_2^*\left( \cO_X(D) \otimes \chi^{\vee} \otimes \Delta \right)$. To see this note that $f(p) g(q) - f(q) g(p)$ is a section of $\pi_1^*\left( \cO_X(D) \right) \otimes \pi_2^* \left( \cO_X(D) \right) \otimes \cO_{X\times X}(-D_X)$, since it vanishes on $D_X$. We know also that $K_{\chi}$ is $\pi_1^*\left( \chi \otimes \Delta \right) \otimes \pi_2^* \left( \chi^{\vee} \otimes \Delta \right) \otimes \cO_{X\times X}(D_X)$. Hence $b_{\chi}$ is a section of the above stated bundle. By K\"{u}nneth formula we know that:
\begin{multline*}
H^0 \left(X \times X, \pi_1^*\left( \cO_X(D) \otimes \chi \otimes \Delta\right) \otimes \pi_2^*\left( \cO_X(D) \otimes \chi^{\vee} \otimes \Delta \right)\right) =\\ H^0\left(X,  \cO_X(D) \otimes \chi \otimes \Delta\right) \otimes H^0 \left(X, \cO_X(D) \otimes \chi^{\vee} \otimes \Delta \right).
\end{multline*}
The matrix $B_{\chi}(f,g)$ is thus obtained from the coefficients of the section with respect to the basis $\mu_{j,k} \otimes \nu_{j^{\prime},k^{\prime}}$. By the definition of $b_{\chi}$ it is immediate that the map $f \wedge g \to B_{\chi}(f,g)$ is well defined and linear.

Now assume that $X$ is real and of dividing type, $D$ is a real divisor and $f$ and $g$ are real meromorphic functions. Then the second equality follows immediately from Corollary \ref{cor:u_times_star}. To see that $B_{\chi}(f,g)$ is $J_v$-Hermitian, we need to prove that $B_{\chi}(f,g)^* = J_v B_{\chi}(f,g) J_v$. Note that from the properties of the Cauchy kernel and the fact that $f$ and $g$ are real meromorphic functions, we get that for every two distinct $p,q \in X$ not in the support of $D$, the equality $\overline{b_{\chi}(f,g)(p,q)} = b_{\chi}(f,g)(q^{\tau},p^{\tau})$. Using this fact for two distinct points not in the support of $D_r$ we get:
\[
 \buc(q)^* B_{\chi}(f,g)^* J_v \buc(p^{\tau}) = \overline{\buc(p^{\tau})^* J_v B_{\chi}(f,g) \buc(q)} = \buc(q)^* J_v B_{\chi}(f,g) \buc(p^{\tau})
\]
Since this is an expression of the same global section with respect to a basis we get the desired equality.
\end{proof}

\begin{rem}
Let $D$ be a divisor on $X$ and $f, g \in \cL(D)$ two meromorphic functions. As in \cite{SV14} one can construct the Bezoutian of $f$ and $g$ using Laurent series expansions in local coordinates. Let $D = \sum_{j=1}^d m_j p_j$ and fix a local coordinate $t_j$ centered at each $p_j$. Let us write $f(t_j(p)) = \sum_{n= -m_j}^{\infty} a_{j,n} t_j(p)^n$ and $g(t_j(p)) = \sum_{n=-m_j}^{\infty} b_{j,n} t_j(p)^n$. Let us write $B_{j,j',k,k'}$ for the entry of $B_{\chi}(f,g)$ corresponding to $\mu_{j,k} \otimes \nu_{j',k'}$. Since $\mu_{j,k}$ has a pole of order $k$ at $j$ with residue one, we see immediately, that what we need is the coefficient of $t_j(p)^{-k} t_{j'}(q)^{-k'}$ in the power series expansion of $b_{\chi}(p,q)$. To this end, we may expand the Cauchy kernel using the coordinate $(t_j,t_j')$ around the origin. If $j \neq j'$ the Cauchy kernel is analytic in a neighborhood of the origin and thus the series will not contain any negative terms, but might still affect some terms of the expansion. It is, however, immediate in this case that $B_{j,j',m_j,m_j'} = \left(a_{j,-m_j} b_{j',-m_{j'}} - a_{j',-m_{j'}} b_{j,-m_j}\right) \frac{K_{\chi}(p_j,p_{j'})}{\sqrt{dt_j}(p_j)\sqrt{dt_{j'}}(p_{j'})}$.
Now assume that $j = j'$. Consider a small open neighborhood of $(p_j,p_j)$ in $X \times X$ trivializing the line bundles we see that $K_{\chi}(p,q) - \frac{1}{p-q}$ is an analytic function and we can write a power series expansion for this function. Taking the product and looking at the principal part we can calculate the entries of $B_{\chi}(f,g)$. For example, note that the term corresponding to $t_j(p)^{-m_j} t_j(q)^{-m_j}$ in the product $f(p)g(q) - f(q)g(p)$ cancels out and we are left with $\frac{a_{j,-m_j+1} b_{j,-m_j} - a_{j,-m_j} b_{j,-m_j+1}}{t_j(p)^{m_j-1} t_j(q)^{m_j-1}} \left( \frac{1}{t_j(q)} - \frac{1}{t_j(p)}\right)$. The analytic terms can only decrease the powers in the denominators, so we are left with $B_{j,j,m_j,m_j} = a_{j,-m_j+1} b_{j,-m_j} - a_{j,-m_j} b_{j,-m_j+1}$ when we multiply by $\frac{1}{t_j(p) - t_j(q)}$.
\end{rem}

Let us assume from now on that $X$ is a Riemann surface of dividing type and fix an orientation on $X(\R)$. Let $D$ be a real and effective divisor and $f,g \in \cL(D)$ be real meromorphic functions. We will also assume that $\chi \in \T_0$ and set $J = J_0 = I_{\deg D_r} \oplus \left(\begin{smallmatrix} 0 & 1 \\ 1 & 0 \end{smallmatrix}\right)^{\oplus \deg D_i}$. Since the Bezoutian is linear and anti-symmetric in $f$ and $g$ we have that for $\alpha,\beta,\gamma,\delta \in \R$
\[
B(\alpha f + \beta g, \gamma f + \delta g) = (\alpha \delta - \beta \gamma) B(f,g).
\]
In particular, this means that choosing $\left( \begin{smallmatrix} \alpha & \beta \\ \gamma & \delta \end{smallmatrix} \right) \in \operatorname{SL}_2(\R)$ the above transformation will not affect the Bezoutian. Using this action we can, whenever it is convenient, assume that $(f)_{\infty} = (g)_{\infty}$. Additionally, note that in the proof of the above result, when we write $b_{\chi}(f,g)$ in terms of the basis of the tensor product, the $\mu_{j,k}$ and $\nu_{j,k}$ will appear if and only if $p_j$ is a pole of $f$ or $g$ and $k$ is at most the maximum of the order of the poles. Therefore, if $(f)_{\infty}, (g)_{\infty} < D$, then $B_{\chi}(f,g)$ will contain a block of zeroes. For this reason, we will assume that either $(f)_{\infty} = D$ or $(g)_{\infty} = D$.

\begin{lem} \label{lem:fiber_basis}
Let $D = (g)_{\infty}$, then for every $\lambda \in \C$, let $(g)_{\lambda} =\sum_{j=1}^r m_j q_j$ be the divisor of the fiber of $g$ at $\lambda$. Fix coordinates $t_j$ around each of the $q_j$, then the set of vectors $\{\dfrac{d^k}{d t_j^k}\buc(q_j) \mid j=1,\ldots,r,\, k=0,\ldots,m_j-1\}$ is a basis for $\C^N$.
\end{lem}
\begin{proof}
Consider the Bezoutian of $g$ and $1$ and applying a linear transformation we may assume that $\lambda = 0$. With respect to $t_j$ ,we have $g(t_j) = a_j t_j^{m_j} + \text{higher order terms}$. Now for $j \neq j'$, we know that:
\[
-\buc(q_j^{\tau})^* J B_{\chi}(g,1) \buc(q_{j'}) = b_{\chi}(g,1)(q_j, q_{j'}) = \left(g(q_j) - g(q_{j'})\right) K_{\chi}(q_j,q_{j'}) = 0.
\]
Fixing $q_{j'}$ and taking the derivative with respect to $t_j$: we get:
\[
- \dfrac{d}{d t_j} \buc(q_j^{\tau})^* J B_{\chi}(g,1) \buc(q_{j'}) = \left(g^{\prime}(q_j) - g(q_{j'})\right) K_{\chi}(q_j,q_{j'}).
\]
The second term is $0$ unless $m_j =1$. Proceeding to take derivatives, we see that the vectors $\dfrac{d^k}{d t_j^k} \buc(q_j^{\tau})$ are orthogonal to $\buc(q_{j'})$ with respect to the form defined by $J B_{\chi}(f,g)$. Similarly, one shows with respect to this form the sets of vectors $\{\dfrac{d^k}{d t_j^k} \buc(q_j^{\tau})\}_{k=1}^{m_j}$ and $\{\dfrac{d^k}{d t_j^k} \buc(q_{j'})\}_{k=1}^{m_{j'}}$ are mutually orthogonal. Therefore, we divide the proof into two cases. First, assume that the point $q_j$ is real. In this case, the subspace spanned by the above sections is orthogonal to all the other subspaces. If $m_j = 1$, we will move on to the next point, hence we assume that $m_j >1$. For $p,q$ close to $q_j$ and distinct, we have that:
\begin{multline*}
b_{\chi}(g,1)(p,q) = \left(a_j (t_j(p)^{m_j} - t_j(q)^{m_j}) + \cdots \right) \left( \frac{1}{t_j(p) - t_j(q)} + \text{analytic terms}\right) \\ = a_j \left( t_j(p)^{m_j-1} t_j(q) + t_j(p)^{m_j-2} t_j(q)^2 + \cdots + t_j(p) t_j(q)^{m_j-1} \right).+ \cdots.
\end{multline*}
If we set $p = q_j$, then the series collapses to $b_{\chi}(g,1)(q_j,q) = a_j t_j(q)^{m_j-1} + \cdots$. Therefore, in particular, taking derivatives with respect to $q$ and substituting $q_j$ for $q$, we get that for $0 \leq k < m_j - 1$, $\buc(q_j)^* J B_{\chi}(f,g) \dfrac{d^k}{d t_j^k}\buc(q_j) = 0$ and since $a_j \neq 0$, then $\buc(q_j)^* J B_{\chi}(g,1) \dfrac{d^{m_j-1}}{d t_j^{m_j-1}}\buc(q_j) = a_j (m_j-1)! \neq 0$. Now to find the other relations we see that taking the derivative with respect to $p$ and substituting $q_j$ instead of $p$ we get:
\[
- \dfrac{d}{d t_j} \buc(q_j)^* J B_{\chi}(g,1) \buc(q) = a_j t_j(q)^{m_j-2} + \cdots
\]
Thus, as above we have that the vector $\dfrac{d}{d t_j} \buc(q_j)$ is orthogonal, with respect to the $J B_{\chi}(g,1)$ inner product, to $\dfrac{d^k}{d t_j^k}\buc(q_j)$ for $0 \leq k < m_j-2$ and is not orthogonal to $\dfrac{d^{m_j-2}}{d t_j^{m_j-2}}\buc(q_j)$. Proceeding in this way we see that the Grammian of the form with respect to this set of vectors has the form:
\[
\begin{pmatrix} 0 & \cdots & 0 & \bullet \\ 0 & \cdots & \bullet & \\0 & \iddots & & \\ \bullet & & & \end{pmatrix}.
\]
Where the numbers on the anti-diagonal are non-zero. This implies that the set of vectors that we obtain at real points are all linearly independent.

If $q_j$ is complex then we consider the subspace spanned by $\dfrac{d^k}{d t_j^k}\buc(q_j)$ and $\dfrac{d^k}{d t_j^k}\buc(q_j^{\tau})$ for $k=0,\ldots,m_j-1$. Let $V = \operatorname{Span}\{\dfrac{d^k}{d t_j^k}\buc(q_j)\}$ and $V_{\tau} = \operatorname{Span}\{\dfrac{d^k}{d t_j^k}\buc(q_j^{\tau})\}$. By the argument, at the beginning of the proof every two vectors in $V$ are orthogonal to each other with respect to $J B_{\chi}(g,1)$ and similarly for every two vectors in $V_{\tau}$. Now let $t_j$ be a coordinate near $q_j$ and $t_j^{\tau}$ a coordinate near $q_j^{\tau}$ and let $p$ be close to $q_j^{\tau}$ and $q$ be close to $q_j$, then as in the real case
\begin{multline*}
b_{\chi}(g,1)(p,q) = \left(a_j (\overline{t_j(p^{\tau})}^{m_j} - t_j(q)^{m_j}) + \cdots \right) \left( \frac{1}{\overline{t_j(p^{\tau})} - t_j(q)} + \text{analytic terms}\right) \\ = a_j \left( \overline{t_j(p^{\tau})}^{m_j-1} t_j(q) + \overline{t_j(p^{\tau})}^{m_j-2} t_j(q)^2 + \cdots + \overline{t_j(p^{\tau})} t_j(q)^{m_j-1} \right).+ \cdots.
\end{multline*}
Now we let $p = q_j^{\tau}$ and as in the real case the series collapses and we find that $\buc(q_j^{\tau})$ is orthogonal with respect to $J B_{\chi}(g,1)$ to every $\dfrac{d^k}{d t_j^k} \buc(q_j)$ for $k=0,\ldots,m_j-2$ and not orthogonal to the last one. Now we can proceed as in the real case to obtain that the Gram matrix on $V \oplus V_{\tau}$ of our inner product has the form:
\[
\begin{pmatrix} 0 & \cdots & \cdots & 0 & 0 & \cdots & 0 & \bullet \\ 0 & \cdots & \cdots & 0 & 0 & \cdots & \bullet & \\0 & \cdots & \cdots & 0 & 0 & \iddots & & \\ 0 & \cdots & \cdots & 0 & \bullet & & & \\ 0 & \cdots & 0 & \bullet & 0 & \cdots  & \cdots & 0 \\ 0 & \cdots & \bullet & & 0 & \cdots  & \cdots & 0 \\0 & \iddots & & & 0 & \cdots & \cdots& 0 \\ \bullet & & & & 0 & \cdots & \cdots & 0\end{pmatrix}.
\]
This implies that our vectors are a basis for $V \oplus V_{\tau}$. Since the Gram matrix for the inner product on the space is the direct sum of these matrices we see that these vectors form a basis.
\end{proof}

Now let $D = (g)_{\infty}$, $N = \deg D$ and let $f \in \cL(D)$. We may assume as above that $(f)_{\infty} = D$ as well. We would like to understand the indefinite inner product defined by $J B_{\chi}(f,g)$. Let $(g)_0 = \sum_{j=1}^r m_j q_j$ and consider the basis $\{\dfrac{d^k}{d t_j^k}\buc(q_j) \mid j=1,\ldots,r,\, k=0,\ldots,m_j-1\}$ obtained in Lemma \ref{lem:fiber_basis}, with respect to a choice of coordinates $t_j$ around $q_j$. Assume that $q_{j_0}$ is a common zero of $f$ and $g$ of multiplicity $\ell$. If $q_{j_0}$ is complex then $q_{j_0}^{\tau}$ is also a common zero of $f$ and $g$ of the same multiplicity. Assume first that $q_{j_0}$ is real. Write $g(t_j(p)) = a_j t_j(p)^{m_j} + \text{higher order terms}$ , where $a_j \neq 0$ and $f(t_j(p)) = b_{\ell} t_j(p)^{\ell} + \text{higher order terms}$, where $b_{\ell} \neq 0$ if $\ell < m_j$ and can vanish otherwise. Assume first that $\ell < m_j$, then for $p, q \in X$ distinct and close to $q_j$ we get:
\begin{multline*}
b_{\chi}(f,g)(p,q) = \\ \left( a_j b_{\ell} t_j(p)^{\ell} t_j(q)^{\ell}( t_j(q)^{m_j -\ell} - t_j(p)^{m_j-\ell}) + \cdots\right) \left( \frac{1}{t_j(p) - t_j(q)} + \text{analytic terms}\right)\\ = a_j b_{\ell} t_j(p)^{\ell} t_j(q)^{\ell}\left( t_j(p)^{m_j - \ell - 1} + t_j(p)^{m_j-\ell-2} t_j(q) + \cdots t_j(q)^{m_j - \ell -1} \right) + \cdots  
\end{multline*}
Taking derivatives with respect to $p$ and $q$ and substituting $q_{j_0}$ we see that for every $0 \leq k < \ell$, the vectors $\dfrac{d^k}{d t_j^k} \buc(q_{j_0})$ are orthogonal to $\dfrac{d^{k'}}{d t_j^{k'}} \buc(q_{j_0})$ for every $0 \leq k' < m_j-1$, with respect to the inner product defined by $J B_{\chi}(f,g)$. Since $b_{\ell} \neq 0$ we see that the inner product of $\dfrac{d^{\ell +k}}{d t_j^{\ell+k}} \buc(q_{j_0})$ with $\dfrac{d^{m_j-1 - k}}{d t_j^{m_j-1-k}} \buc(q_{j_0})$ is $a_j b_{\ell} (\ell + k)! (m_j - 1 - k)! \neq 0$ for $k = 0,\ldots, m_j-1 - \ell$. Note that as in the proof of Lemma \ref{lem:fiber_basis} the vectors corresponding to different points in the support of $(g)_0$ are orthogonal with respect to our indefinite inner product. We can conclude that the first $\ell$ vectors corresponding to $q_{j_0}$ are orthogonal to every vector. If $q_{j_0}$ is complex, we consider the spaces $V \oplus V_{\tau}$ as in Lemma \ref{lem:fiber_basis} to see that both for $q_{j_0}$ and $q_{j_o}^{\tau}$ the first $\ell$ vectors are orthogonal to the entire space with respect to $J B_{\chi}(f,g)$. This implies that $\dim \ker J B_{\chi}(f,g) = \dim \ker B_{\chi}(f,g) \geq \deg \left((f)_0 \wedge (g)_0 \right)$. On the other hand, if $v \in \ker B_{\chi}(f,g)$, then we can write $v$ as a linear combination of vectors in our basis and taking the indefinite inner product with the vectors of our basis in increasing order we see that $v$ must be spanned by the degenerate vectors. Therefore, we have proved:

\begin{lem} \label{lem:Bezoutian_nullity}
Let $D = (g)_{\infty}$ and $f \in \cL(D)$, then
\[
\dim \ker J B_{\chi}(f,g) = \dim \ker B_{\chi}(f,g) = \deg \left( (f)_0 \wedge (g)_0 \right).
\]
\end{lem}

Now in order to understand this form better we need to consider the meromorphic function $h = \frac{f}{g}$. As observed above, applying a M\"{o}bius map in $\operatorname{SL}_2(\R)$ we may assume that $D = (f)_{\infty} = (g)_{\infty}$. We obtain thus that $\deg h = N - \deg \left((f)_0 \wedge (g)_0\right) = \rank B_{\chi}(f,g)$ and $(h)_{\infty} = (g)_0 - (f)_0 \wedge (g)_0$. We will use $h$ to construct a different set of vectors, not in the kernel of $B_{\chi}(f,g)$. Let $\lambda \in \R$ be such that $h^{-1}(\lambda) \cap \left(\Supp D \cup \Supp(g)_0\right) = \emptyset$ and assume that $h$ is unramified over $\lambda$. Then we have that for two distinct points $p,q \in h^{-1}(\lambda)$:
\[
b_{\chi}(f,g)(p,q) = g(q) g(p) \left(h(p) - h(q) \right) K_{\chi}(p,q) = 0.
\]
Let us write $(h)_{\lambda} = \sum_{j=1}^m p_j$, with all $p_j$ distinct since the fiber is unramified. Furthermore, since $h$ is real and $\lambda \in \R$, the divisor $(h)_{\lambda}$ is real. Let us fix coordinates $t_j$ at each $p_j$ and assume that the coordinates are real and respect the orientation for real points and respect $\tau$ for conjugate points. Thus we have obtained that $\buc(p_j^{\tau})$ and $\buc(p_{j'})$ are orthogonal with respect to our inner product for $j \neq j'$. Now fix $p_j$ and note that with respect to $t_j$ we have that $h'(t_j(p)) = \frac{f'(t_j(p)) g(t_j(p)) - f(t_j(p)) g'(t_j(p))}{g(t_j(p))^2}$ and the derivative does not vanish at $p_j$ since it is not a branch point. On the other hand, we know that $b_{\chi}(f,g)(p_j,p_j) = -h'(p_j) g(p_j)^2 \neq 0$, therefore $\buc(p_j)$ is non-isotropic for $p_j$ real and isotropic otherwise. However, for $p_j \neq p_j^{\tau}$  the inner product of $\buc(p_j)$ with $\buc(p_j^{\tau})$ is non-zero. Furthermore, we see that if $p_j \in X(\R)$, then the inner product of $\buc(p_j)$ with itself with respect to $J B_{\chi}(f,g)$ is a real number with the same sign as $h'(p_j)$. We summarize this section in the following theorem:
\begin{thm} \label{thm:signature}
Let $X$ be a real Riemann surface of dividing type and fix an orientation on $X(\R)$. Let $D$ be a real divisor on $X$, $f, g \in \cL(D)$ be real meromorphic functions be such that $(g)_{\infty} = D$ and set $h = f/g$. Let $\chi \in \T_0$ and $J$ be the signature matrix of $D$ with respect to $\chi$. Let $\lambda \in \R$ be such that the fiber of $h$ over $\lambda$ is unramified and does not contain zeroes of $g$ or points from the support of $D$. Let $c$ be the number of complex conjugate pairs in $h^{-1}(\lambda)$, $r_+$ the number of real points in $h^{-1}(\lambda)$, such that the derivative of $h$ at them preserves orientation and $r_-$ the number of real points where the orientation is reversed. Then the inertia of the self-adjoint matrix $J B_{\chi}(f,g)$ is:
\begin{itemize}
\item $n_+ = r_+ + c$;
\item $n_- = r_- + c$;
\item $n_0 = \deg D - \deg \left( (f)_0 \wedge (g)_0 \right)$.
\end{itemize}
In particular, the signature of $J B_{\chi}(f,g)$ is $r_+ - r_-$.
\end{thm}

\section{Main Result} \label{sec:main_result}

Let $X$ again be a real Riemann surface of dividing type with a fixed orientation on $X(\R)$. Let $D$ be a real divisor on $X$ and $f,g \in \cL(D)$ be real meromorphic functions. As in the previous section, we will assume that $(f)_{\infty} = (g)_{\infty} = D$. Denote again $h = f/g$. Fix a flat line bundle $\chi \in \T_0$, $\Delta$ a line bundle of half-order differential on $X$ and let $J$ be the signature matrix for $D$ with respect to $\chi$. In the previous section, we have discussed Bezoutians on Riemann surfaces and signatures of the self-adjoint matrices associated to them. In this section, we will use the tool of the signature to collect topological data on the map $h$. We denote by $\sigma J B_{\chi}(f,g)$ the signature of the matrix.
\begin{thm} \label{thm:from_sig_to_coh}
\[
\sigma J B_{\chi}(f,g) =  \sum_{Y \subset X(\R)} \deg h_*|_{H^1(Y)}([Y]).
\]
The sum runs over the circles in $X(\R)$ and takes the winding number of $h$ restricted to this circle ($0$ if the circle does not cover $\pp^1(\R)$).
\end{thm}
\begin{proof}
Fix $\lambda \in \R$, such that the fiber of $h$ over $\lambda$ is unramified, and does not contain any points from the support of $D$ and zeroes of $g$. Let $Y \subset X(\R)$ be one of the circles. If $h^{-1}(\lambda) \cap Y = \emptyset$, then $Y$ does not cover $\pp^1(\R)$ and thus $h_*([Y]) = 0$. Assume now that $h^{-1}(\lambda) \cap Y = \{p_1,\ldots,p_k\}$ ordered according to the orientation induced on $Y$ from $X(\R)$. Let us proceed on the segments between $p_i$ and $p_{i+1}$ and observe the change in the sign of the derivative. If the derivative of $h$ changes signs, then it is not a covering for this implies that there is a point, at which as we move along the segment the image changes direction and goes back. If the derivative doesn't change signs it means that we arrive at the same point in the image from the other direction thus completing a loop.

\end{proof}

\begin{rem}
This, in fact, is a version of the Cauchy index of a real rational function on the entire real line. In \cite[X]{KreNai81} it is proved that the signature of the Bezoutian of two real polynomials is precisely the Cauchy index of their quotient and the above theorem generalizes this result to the setting of Riemann surfaces.
\end{rem}

This leads to the following result, which is a generalization of the argument principle to the case at hand.
\begin{thm} \label{thm:argument}
Let $f$ and $g$ be real meromorphic functions on $X$ and $\lambda \in \C_+$, then:
\begin{equation} \label{eq:argument}
\sigma J B_{\chi}(f,g) = \sum_{p \in h^{-1}(\lambda) \cap X_+} e(p) - \sum_{p \in h^{-1}(\lambda) \cap X_-} e(p).
\end{equation}
Here $e(p)$ stands for the multiplicity of the point $p$ in the fiber over $\lambda$.
\end{thm}
\begin{proof}
We will provide two proofs for this fact, one relying on Theorem \ref{thm:from_sig_to_coh} and another that is independent of it. We first note that for every point $\lambda \in \C_+$ the right-hand side is constant. Indeed, if for $\lambda_1, \lambda_2 \in \C_+$ the right-hand side was distinct, then draw a line between them in $\C_+$ and look at the curves in the preimage of this line under $h$. One of the curves starts in $X_+$ and end in $X_-$ and since $X$ is dividing it must cross $X(\R)$. However, $h$ is real and thus maps the real point onto the real axis and this is a contradiction. 

\textbf{First Approach:} Let $\mu \in \R$, such that the fiber of $h$ is unramified over $\mu$ and does not contain points from the support of $D$ or $(g)_0$. We will perturb $\mu$ slightly and observe the changes in the fiber. Choose a small disc $D$ around $\mu$, such that $h|_{h^{-1}(D)}$ is a covering map. Let us write $h^{-1}(D) = \sqcup_{p \in h^{-1}(\mu)} D_p$, where each $D_p$ is an open subset of $X$, containing $p$ and homeomorphic to $D$ via $h$. If $p \in h^{-1}(\mu) \cap X_+$, then, shrinking $D$ if necessary, we may assume that $D_p \subset X_+$. Hence every point in $D$ has a preimage in $X_+$. However, applying the same argument to $p^{\tau} \in h^{-1}(\mu) \cap X_-$ we see that each point in $D$ has a preimage in $X_-$. Furtheremore, for every $q \in D_p$, $h(q^{\tau}) = \overline{h(q)} \in D$, therefore $D_p^{\tau} \subset D_{p^{\tau}}$ and applying the same argument in reverse we see that $D_p^{\tau} = D_{p^{\tau}}$. Therefore, for every pair of conjugate preimages of $\mu$, every point in $D$ has a pair of conjugate preimages. In particular, for $\lambda \in D \cap \C_+$, the conjugate pair does not contribute to the right-hand side of \eqref{eq:argument}. Now for $p \in h^{-1}(\mu) \cap X(\R)$. Again shrinking $D$ if necessary, we may assume that $D_p$ is equipped with a real coordinate $t$ respecting the orientation of $X(\R)$, centered at $p$. Let us write the Taylor series expansion of $h$ with respect to $t$, $h(t(q)) = \mu + h^{\prime}(t(p)) t(p) + \cdots$. Now note that by our assumption on $t$, for $q \in D_p \cap X_+$ we have that $t(q) \in \C_+$. Thus if $h^{\prime}(t(p)) > 0$, then for every $q \in D_p \cap X_+$ close enough, we get that $h(t(q))$ is a real number plus a number in $\C_+$ plus a small error term. Therefore, in particular, every $q \in D_p \cap X_+$ close enough to $p$ is mapped to $\C_+$. Similarly, if $h^{\prime}(t(p))$ is negative, then every $q \in D_p \cap X_+$ close enough to $p$ is mapped to $\C_-$. Shrinking $D$ even further, we conclude that for every $\lambda \in \C_+ \cap D$, if $h^{\prime}(t(p)) >0$, $\lambda$ has a preimage in $X_+$ and if $h^{\prime}(t(p)) < 0$, then $\lambda$ has a preimage in $X_-$. By Theorem \ref{thm:signature} we know that $\sigma J B_{\chi}(f,g)$ is precisely the difference between the number of points $p \in h^{-1}(\mu) \cap X(\R)$, such that the derivative of $h$ with respect to $t$ is positive and the number of points where the derivative is negative. Therefore, we obtain \eqref{eq:argument}.

\textbf{Second Approach:} Consider a new function on $X$ defined by $\varphi = \frac{h -\lambda}{h - \overline{\lambda}}$. Note that $\varphi$ is a composition of $h$ with a M\"obius map, that maps the upper half-plane onto the disc. Define a meromorphic differential with simple poles on $X$ by $\omega = \frac{d \varphi}{\varphi}$. For each circle $Y \subset X(\R)$ we have that $\frac{1}{2 \pi i} \int_Y \omega$ is precisely the winding number of $\varphi(Y)$ around the origin. By Theorem \ref{thm:from_sig_to_coh} we know that the sum of the winding numbers is precisely $\sigma J B_{\chi}(f,g)$. On the other hand, $\frac{1}{2 \pi i} \int_Y \omega$ is the number of zeroes of $\varphi$ in $X_+$ counting multiplicities minus the number of poles of $\varphi$ in $X_+$ counting multiplicities. To see this note first that the poles of $\omega$ are precisely the zeroes and poles of $\varphi$ and that the residues are the order. Let $U$ be an open neighborhood of $\overline{X_+}$, where there are no additional zeroes or poles of $\varphi$, except for those in $X_+$ itself. The form $\omega$ is holomorphic on $U$ without the poles, therefore it is closed and we can replace the integral by the integrals on small circles around each pole of $\omega$. Now to conclude the proof note that the poles of $\varphi$ are precisely the points in the fiber of $h$ over $\overline{\lambda}$ and the zeroes are the fiber of $h$ over $\lambda$. Therefore, we get that
\[
\sigma J B_{\chi}(f,g) = \sum_{p \in h^{-1}(\lambda) \cap X_+} e(p) - \sum_{p \in h^{-1}(\overline{\lambda}) \cap X_+} e(p).
\]
Since $h$ is real we have that $h(p) = \lambda$ if and only if $h(p^{\tau}) = \overline{\lambda}$. Hence, $p \in h^{-1}(\overline{\lambda}) \cap X_+$ if and only if $p^{\tau} \in h^{-1}(\lambda) \cap X_-$ and thus $\sum_{p \in h^{-1}(\overline{\lambda}) \cap X_+} e(p) = \sum_{p \in h^{-1}(\lambda) \cap X_-} e(p)$. This concludes the proof.
\end{proof}

\begin{rem} \label{rem:dividing}
Recall from \cite{SV14} that a real meromorphic function $h$ on a real Riemann surface of dividing type is called dividing if $p \in X(\R)$ if and only if $h(p) \in \R \cup \{\infty\}$. In \cite[Proposition 5.7]{SV14} it is proved that if $h = f/g$, with $f$ and $g$ real meromorphic functions with simple real poles, then the Bezoutian is definite. The above theorem allows us to generalize this to general real $f$ and $g$. Namely, let $h = f/g$, with $f$ and $g$ real, if $h$ is dividing, then $h(X_+) = \C_+$ or $h(X_+) = \C_-$, thus $\sigma J B_{\chi}(f,g) = \deg h$ or $\sigma J B_{\chi}(f,g) = - \deg h$. Since $\deg h$ is the rank of this Hermitian matrix we see that $J B_{\chi}(f,g)$ is respectively positive or negative semidefinite. Conversely, if $J B_{\chi}(f,g)$ is without loos of generality positive semi-definite, then by the above theorem, the preimage of $\C_+$ is in $X_+$, but since the function is real, they have to coincide. Conclude that $h$ is dividing.
\end{rem}

We need a few corollaries of the above theorem in order to apply the theory of Bezoutians to polynomials on quadrature domains. To every meromorphic function $f$ on $X$, we associate two real meromorphic functions, the real part $f_r = \frac{1}{2} (f + f^{\tau})$ and the imaginary part $f_i = \frac{1}{2 i}(f - f^{\tau})$. Let $D_f = (f)_{\infty} \vee (f^{\tau})_{\infty}$. It is clear that $D_f$ is a real divisor and that $f_r, f_i \in \cL(D)$. Furthermore, note that if $p \in X(\R)$ is a pole of $f$ of order $m$, then $p$ is a pole of $f^{\tau}$ of the same order and thus $p$ is a pole of both $f_r$ and $f_i$ of order at most $m$. Furthermore, if $p$ is a pole of $f_r$ of order less than $m$, then $p$ is a pole of $f_i$ of order precisely $m$ and vice verse. Assume now that $p$ and $p^{\tau}$ are both poles of $f$ of orders $m$ and $n$, respectively, Then if $m \neq n$, then both points are poles of $f_r$ and $f_i$ of order $\max\{m,n\}$. If $m = n$, then as in the case of a real pole if there is cancellation in $f_r$ and $p$ is a pole of order less than $m$ of $f_r$, then it is a pole of order precisely $m$ of $f_i$ and vice verse. The same applies to $p^{\tau}$. Thus $D_f = (f_r)_{\infty} \vee (f_i)_{\infty}$. From now on the Bezoutians will be calculated with respect to the divisor $D_f$.
\begin{cor} \label{cor:cayley}
Let $f$ be a meromorphic function on $X$ and set $F = \frac{f^{\tau}}{f}$, then
\[
\sigma J B(f_r,f_i) = \sum_{p \in F^{-1}(0) \cap X_+} e(p) - \sum_{p \in F^{-1}(0) \cap X_-} e(p)
\]
Furthermore, $\dim \ker B_{\chi}(f_r,f_i) = (f)_0 \wedge (f^{\tau})_0$.
\end{cor}
\begin{proof}
Let $h = \frac{f_r}{f_i} = i \frac{1 + F}{1 - F}$. The last equality implies that $h$ is the composition of $F$ with the Cayley transform that takes the disc to the upper half-plane and the origin to $i$. By Theorem \ref{thm:argument} we know that:
\[
\sigma J B_{\chi}(f_r, f_i) = \sum_{p \in h^{-1}(i)\cap X_+} e(p) - \sum_{p \in h^{-1}(i) \cap X_-} e(p).
\]
This proves the first claim. To see the second claim note that each $p \in X(\R)$ that is a zero of $f$ of order $m$ is also a zero of $f^{\tau}$ of order $m$ and thus is a common zero of $f_r$ and $f_i$ of order $m$. If $p \in X \setminus X(\R)$ is a common zero of $f$ and $f^{\tau}$ it implies that both $p$ and $p^{\tau}$ are zeroes of $f$ and letting $m = \min\{\ord_p f, \ord_{p^{\tau}} f\}$ we see again that both $p$ and $p^{\tau}$ are common zeroes of $f_r$ and $f_i$ of order $m$. Applying Theorem \ref{thm:signature} we get the second claim.
\end{proof}

Now assume that $f$ has no poles in $X_+$. Therefore, the common poles of $f$ and $f^{\tau}$ are real and the rest are disjoint and conjugate to each other. Note that for every $p \in X$ we have that $\ord_p(F) = \ord_{p^{\tau}}(f) - \ord_p(f)$. In particular, the only possible zeroes of $F$ in $X_+$ are the zeroes of $f$ in $X_-$ that are not canceled by conjugate zeroes in $X_+$. Let us associate four numbers with $f$. Let $a_+$ be the number of zeroes $p$ of $f$ in $X_+$ counting multiplicity, such that $p^{\tau}$ is not a zero of $f$. Similarly, let $a_-$ be the number of zeroes of $f$ in $X_-$ counting multiplicity, such that their conjugate is not a zero of $f$. Let $b_+$ be the number of zeroes $p$ of $f$ in $X_+$, such that $p^{\tau}$ is also a zero of $f$, but $\ord_p f > \ord_{p^{\tau}} f$ counted with the multiplicity $\ord_p f - \ord_{p^{\tau}} f$. Let $b_-$ be the same quantity defined for $X_-$. Let us write $\pi$ and $\nu$ for the number of positive and negative eigenvalues of $J B_{\chi}(f_r,f_i)$, respectively. The following two corollaries describe the two extreme cases for the degree of $D_f$, namely $\deg D_f = \deg f$ and $\deg D_f = 2 \deg f$.
\begin{cor} \label{cor:classic}
Let $f$ be a meromorphic function on $X$ and assume that all of the poles of $f$ are real, then $D_f = (f)_{\infty}$. Then $a_+ + b_+ = \nu$. If in addition, $f$ has no zeroes on $X(\R)$, then the number of zeroes of $f$ in $X_+$ is $\nu + \frac{n_0}{2} = \frac{\deg f - \sigma}{2}$.
\end{cor}
\begin{proof}
Since all of the poles of $f$ are real they do not contribute to the zeroes of $F$ and thus the number of zeroes of $F$ in $X_+$ is $a_- + b_-$ and the number of zeroes in $X_-$ is $a_+ + b_+$. Therefore, by the first part of Corollary \ref{cor:cayley} we see that:
\[
\sigma = a_- + b_- - a_+ - b_+.
\]
Additionally, $\deg f = a_- + b_- + a_+ + b_+ + n_0$, where $n_0 = \dim \ker B_{\chi}(f_r,f_i)$, by the second part of Corollary \ref{cor:cayley}.  Subtracting the two equation we get that $\deg f - \sigma = 2 a_+ + 2 b_+ + n_0$. Now since the size of the matrix $J B_{\chi}(f_r,f_i)$ is $\deg D_f = \deg f$, we see that $\sigma = \pi - \nu$ and $\deg f = n_0 + \pi + \nu$. Therefore, we conclude that $a_+ + b_+ = \nu$.

If $f$ has no zeroes on $X(\R)$, then $n_0$ is even, since the common zeroes of $f$ and $f^{\tau}$ come in conjugate pairs. Therefore, the number of zeroes of $f$ in $X_+$ is precisely $a_+ + b_+ + \frac{n_0}{2}$ and we get the second formula in the claim.
\end{proof}
\begin{cor} \label{cor:apply}
Let $f$ be a meromorphic function on $X$ that has no poles in $\overline{X_+}$, then $a_+ + b_+ = \nu - \deg f$. If in addition, $f$ has no zeroes on $X(\R)$, then the number of zeroes of $f$ in $X_+$ is $\frac{- \sigma}{2}$.
\end{cor}
\begin{proof}
We proceed as in the proof of the previous corollary. However, $\deg D_f = 2 \deg f$ and since the poles of $f$ are concentrated in $X_-$, they all contribute to the zeroes of $F$. Hence we get that
\[
\sigma = a_- + b_- - a_+ - b_+ - \deg f.
\]
Now again $\deg f = a_+ + b_+ + a_- + b_- + n_0$ and subtracting we get that $-\sigma = 2 a_+ + 2 b_+ + n_0$. However , in this case $2 \deg f = \pi + \nu + n_0$ and since $\sigma = \pi - \nu$ ,we have that 
\[
a_+ + b_+ = \frac{- \sigma - n_0}{2} = \frac{- \pi - \nu -n_0 + 2 \nu}{2} = \nu - \deg f.
\]
The second formula follows as in the preceding proof.
\end{proof}

\section{Application to Quadrature Domains} \label{sec:quadrature}

A quadrature domain $U \subset \C$ is a domain, such that there exists finitely many points $a_1,\ldots,a_k \in U$ , non-negative integers $m_1,\ldots,m_k$ and constants $c_{ij}$ for $1 \leq i \leq k$, $0 \leq j \leq m_i$, such that for every function analytic and absolutely integrable in $U$, we have that $\int_{U} f = \sum_{i=1}^d \sum_{j=0}^{m_i} c_{ij} f^{(j)}(a_i)$. As mentioned in the introduction, there exists a meromorphic function on the Schottky double $X$ of $U$, that maps $X_+$ conformally onto $U$. This can be also viewed as extending the coordinate function $z$ on $U$ to a meromorphic function on $X$. We will denote this meromorphic function by $\varphi$. Fix $\chi \in \T_0$ and $\Delta$ a line bundle of half-order differentials on $X$ as above. Let $D_{\varphi} = (\varphi)_{\infty} \vee (\varphi^{\tau})_{\infty}$. Since the poles of $\varphi$ are in $X_-$, then $D_{\varphi} = (\varphi)_{\infty} + (\varphi^{\tau})_{\infty}$. Let $p$ be a polynomial of degree $n$. We would like to be able to count the number of zeroes of $p$ in $U$. To this end, we note that $p(\varphi)$ is a meromorphic function on $X$ and we can use the results of Section \ref{sec:main_result} to fulfill our goals. Since for every meromorphic function on $X$ we have $f = f_r + i f_i$, by the properties of the Bezoutian we have that $B_{\chi}(f,f^{\tau}) = 2 i B_{\chi}(f_r,f_i)$. Let $p(z) = \sum_{j=0}^n a_j z^j$, then $B_{\chi}(p(\varphi),p(\varphi)^{\tau}) = \sum_{j,\ell =0}^n a_j \overline{a_{\ell}} B_{\chi}(\varphi^j,\varphi^{\tau \ell})$. Here the Bezoutians are taken with respect to the divisor $n D_{\varphi}$ and whenever both $j$ and $\ell$ are less than $n$, we complement by zeroes. Lastly, we note that $J$ consists just of blocks of the form $\left( \begin{smallmatrix} 0 & 1 \\ 1 & 0 \end{smallmatrix} \right)$. The following result follows readily from Corollary \ref{cor:apply}.

\begin{thm} \label{thm:count_quadrature}
Let $p = \sum_{j=0}^n a_j z^j$ be a polynomial and $U$ a quadrature domain with Schottky double $X$. Let $\varphi$ be the meromorphic function that identifies $X_+$ with $U$ and $D_{\varphi} = (\varphi)_{\infty} + (\varphi^{\tau})_{\infty}$. Then $p(\varphi)$ and $p(\varphi^{\tau})$ have $\dim \ker B_{\chi}(p(\varphi),p(\varphi^{\tau}))$ zeroes in common counting multiplicities. Moreover, $p$ has additional $\nu - n \deg \varphi$ zeroes in $U$, where $\nu$ is the number of negative eigenvalues of $\frac{1}{2i} J B_{\chi}(p(\varphi),p(\varphi^{\tau}))$. If additionally, $p$ has no zeroes on $\partial U$, then the number of zeroes in $U$ is $-\sigma/2$, where $\sigma$ is the signature of $\frac{1}{2i} J B_{\chi}(p(\varphi),p(\varphi^{\tau}))$. We can compute the matrix $B_{\chi}(p(\varphi),p(\varphi^{\tau}))$ by just knowing the matrices $B_{\chi}(\varphi^j,\varphi^{\tau \ell})$ for $0 \leq j,\ell \leq n$.
\end{thm}

Next, we will use this theorem to calculate the number of zeroes of some polynomials in simply connected quadrature domains. We intend to pursue explicit computations of the higher genus case in future work.


\subsection{Domains of Genus 0} \label{sec:genus_0}

In this section, we will give examples of simply connected quadrature domains and find the number of zeroes of polynomials in them using the methods described above. First, we note that since the domain is simply connected, $X$ is of genus $0$ and thus is a Riemann sphere. The Cauchy kernel on the Riemann sphere is $K(t,s) = \frac{\sqrt{dt} \sqrt{ds}}{t - s}$. Now the Cauchy kernel and its derivatives at any point $\lambda \in \C$ are simply $\frac{\sqrt{dt}}{t - \lambda},\frac{\sqrt{dt}}{(t - \lambda)^2},\ldots, \frac{(n-1)! \sqrt{dt}}{(t - \lambda)^n}$. Now given a polynomial $f$ we have that $f^{\tau}(t) = \overline{ f( \overline{z} )}$.

In the case of genus $0$, the function $\varphi$ is a rational function mapping the upper half-plane conformally onto the quadrature domain. If the poles of $\varphi$ are $a_1,\ldots,a_r$ with multiplicities $m_1,\ldots,m_r$, respectively, then we can write $\varphi(t)^j = \frac{P(t)}{Q(t)} = \sum_{p=1}^r \sum_{q=1}^{j m_p} \frac{\alpha_{pq}}{(t - a_p)^q}$. Thus the divisor of poles $D = j \sum_{p=1}^r m_p a_p + j \sum_{p=1}^r m_p \overline{a_r}$. We now compute
\begin{multline*}
b(\varphi^j, \varphi^{\tau \ell}) = \frac{\frac{P^j(t)}{Q^j(t)} \frac{\overline{P^{\ell}(\overline{s})}}{\overline{Q^{\ell}(\overline{s})}} - \frac{P^j(s)}{Q^j(s)} \frac{\overline{P^{\ell}(\overline{t})}}{\overline{Q^{\ell}(\overline{t})}}}{t-s} = \sum_{p,q,p',q'} \frac{\frac{\alpha_{pq} \overline{\alpha_{p'q'}}}{(t - a_p)^q)(s - \overline{a_{p}})^{q'}} - \frac{\alpha_{pq} \overline{\alpha_{p'q'}}}{(s - a_p)^q)(t - \overline{a_{p}})^{q'}}}{t - s} = \\
\sum_{p,q,p',q'} \alpha_{pq} \overline{\alpha_{p'q'}} \frac{(s - a_p)^q)(t - \overline{a_{p}})^{q'} - (t - a_p)^q)(s - \overline{a_{p}})^{q'}}{(t-s)(s - a_p)^q)(t - \overline{a_{p}})^{q'})(t - a_p)^q)(s - \overline{a_{p}})^{q'}}
\end{multline*}
Therefore, the calculation can be reduced to the partial fraction decomposition of each summand. Note that the classical Bezoutian appears in the numerator once $t-s$ is canceled from the denominator. We will call the matrix $\frac{1}{2i} J B_{\chi}(p,p^{\tau})$, the Bezoutian or the Bezout matrix associated to $p$. 

A program\footnote{The code can be obtained from \url{https://uwaterloo.ca/scholar/eshamovi/software}} that calculates the Bezoutian associated to a polynomial with respect to a quadrature domain $U$ described by a function $\varphi$ mapping conformally the upper half-plane onto $U$ was implemented in Maple\footnote{Maple is a trademark of Waterloo Maple Inc.} \cite{maple}. The matrices appearing in the following subsections are obtained using this program. 

\begin{ex} \label{ex:disc}
The simplest and chronologically first example of a quadrature domain is the disc. Although one can establish the stability of a given polynomial with respect to the unit disc using the Schur-Cohn theory (see \cite[Theorem 43.1]{Mar66}), we will use it as our first example for its simplicity. We take $\varphi$ to be the Cayley transform $\varphi(z) = \frac{z - i}{z + i}$. Calculating the Bezoutian is easy since the only expressions that appear are of the form:
\[
\frac{(s-i)^j (t + i)^{\ell} - (t - i)^j (s + i)^{\ell}}{(t-s)(s-i)^j (t - i)^j (s +i)^{\ell} (t+i)^{\ell}}.
\]

Let $p(z) = z^4 + 5 z^3 -2 z^2 + 3 z -4$, then this polynomial has $3$ roots in $\D$ and one outside. We have that $\frac{1}{2i} J B_{\chi}(p,p^{\tau})$ is:
\[
\left( \begin {array}{cccccccc} - 407 &
 192& - 888\,i& -
 432\,i& 960 &-
 528& 224 & 
 128\,i\\ \noalign{\medskip} 192&- 408& 432\,i& 888\,i&- 528 &
 960& - 128\,i& -
 224\,i\\ \noalign{\medskip} 888
\,i& - 432\,i&- 2736&-
 528& - 3040\,i& 
 128 \,i& 704 & 0
\\ \noalign{\medskip} 432\,i& -
 888 \,i&- 528&-
 2736& -128 \,i& 3040
\,i& 0 & 704\\ \noalign{\medskip}
 960 &- 528 & 
 3040\,i& 128.\,i&-
 4160 & 0 & -
 1024\,i& 0
\\ \noalign{\medskip}- 528 &
 960 & - 128 \,i&  - 3040
\,i& 0 &- 4160 &
  0 & 1024 \,i
\\ \noalign{\medskip} - 224\,i& 128\,i&
 704 & 0& 1024\,i& 0 &- 256 & 0 \\ \noalign{\medskip} -
 128\,i& 224\,i& 0 & 704 & 0 & - 1024
\,i& 0 &- 256 \end {array}
 \right) 
\] 
The inertia of this matrix is $n_+ = 1$, $n_- = 7$ and $n_0 = 0$. As observed above there are no roots on the boundary and $3 = 7 - 4 \cdot 1$ roots in the disc, as predicted by Theorem \ref{thm:count_quadrature}.

\end{ex}

\begin{ex} \label{ex:double_pole}
In \cite[Theorem 4]{AhaSha76} it is shown that a quadrature domain that satisfies the quadrature identity with a single point and order two is simply connected. Assuming that the point is the origin, Ahronov and Shapiro prove that the boundary is given by the polynomial $P(z,w) = z^2 w^2 - a z w -b (z + w) -c$, where $a,b,c >0$ are constants and the domain itself is the image of the unit disc under a quadratic polynomial, that is univalent on $\D$. In fact, this domain is bounded by the classical cardioid. If we take the polynomial to be $z^2 + 3 z$, then $P(z,w) = z^2 w^2 - 11 z w - 9(z + w) -8$. So $\varphi(z) = \frac{(z-i)^2}{(z+i)^2} + 3 \frac{z-i}{z+i}$. Furthermore, $U$ itself is the set of points, such that $P(z,\overline{z}) < 0$. 

Let $p = z^4 + 5 z^3 -2 z^2 + 3 z -4$ again, then $\frac{1}{2i} J B_{\chi}(p,p^{\tau})$ is a $16 \times 16$ matrix with inertia $n_+ = 5$, $n_- = 11$ and $n_0 = 0$. Therefore, the formula tells us that the polynomial should have $11 - 8 = 3$ roots in this domain. Using the Maple polynomial solver it easily verified that indeed three of the roots satisfy $P(z,\overline{z}) < 0$ and one does not. Unfortunately, the matrix is too big to be presented in the text, therefore let us take $q(z) = z^2 -4$ and observe that it has one root on the boundary of our quadrature domain and another inside. The associated matrix is:
\[
\left( \begin {array}{cccccccc} - 2248 & 328& -
 4152 i& - 984 i& 2560&- 640& 512 i& 128 i\\ \noalign{\medskip} 328 &- 2248
& 984 i&  4152 i&- 640& 2560& -
 128 i& - 512 i\\ \noalign{\medskip} 4152 i& -
 984 i&- 10256 &- 640& - 7168 i& 
 128 i& 1536 & 0 \\ \noalign{\medskip} 984 i& - 4152 i&- 640&- 10256& - 128 i&
 7168 i& 0& 1536\\ \noalign{\medskip}
 2560&- 640& 7168 i& 128 i&-
 5248& 0&- 1152 i& 0
\\ \noalign{\medskip}- 640 & 2560 & - 128 i&
 - 7168 i& 0&- 5248 & 0&
 1152 i\\ \noalign{\medskip} - 512 i& 128 i& 1536 & 0 & 1152 i& 0 &- 256 & 0 \\ \noalign{\medskip} - 128 i& 512 i& 0 & 1536 & 0 & - 1152 i& 0 &- 256 \end {array} \right)
\]
This matrix has kernel of dimension $1$ and $n_- = 5$ as expected.

The following figure shows the disc, the cardioid and the zeroes of $p$:

\begin{ex} \label{ex:neumann}
A Neumann oval is the reflection of an ellipse with respect to the circle. It is a quadrature domain with quadrature identity of order two, with two distinct nodes (see \cite{Gus83} for a discussion of this case). In particular, one can obtain a Neumann oval as the image of the upper half-plane under the map $\varphi(z) = \frac{15(z^2 + 1)}{6 z^2 + 20i z - 6}$. It is clear that the map is of degree $2$. A straightforward calculation shows that $\varphi$ is indeed univalent on the upper half-plane.

The Bezoutian of $p$ in this case is again an invertible $16\times 16$ matrix and the number of negative eigenvalues of $\frac{1}{2i} J B_{\chi}(p(\varphi(z)),\overline{p(\varphi(\bar{z}))})$ is $11$. the number of roots of $p$ inside the oval is $3$. As in Example \ref{ex:double_pole} the size of the matrix prevents us from presenting the Bezout matrix associated to $p$ and instead we will present the matrix associated to $z^2 -4$
\[
\begin{small}\left( \begin {array}{cccccccc} - 10.579& 3.958&
 3.255 i& - 4.167 i& 8.965& 0 & 
 24.61 i& 0 \\ \noalign{\medskip} 3.958 &-
 10.579 &  4.167 i&  - 3.255 i& 0 &
 8.965 & 0 &  - 24.61 i\\ \noalign{\medskip}
  - 3.255 i&  - 4.167 i&- 1.447 & 0 &  
 10.547 i& 0 &- 23.438 & 0 
\\ \noalign{\medskip}  4.168 i&  3.255 i& 0 &-
 1.447 & 0 &  - 10.547 i& 0 &-
 23.438 \\ \noalign{\medskip} 8.965 & 0 &
  - 10.547 i& 0  &- 388.183 & 189.375 &
  - 615.234 i&  - 337.5 i\\ \noalign{\medskip} 0 & 8.965 & 0 &  10.547 i& 189.375 &-
 388.183 &  337.5 i&  615.234 i
\\ \noalign{\medskip}  - 24.61 i& 0 &- 23.438 
& 0 &  615.234 i&  - 337.5 i&- 1054.688 & 0 \\ \noalign{\medskip} 0 &  24.61 i&
 0 &- 23.438 &  337.5 i&  - 615.234 i&
 0 &- 1054.688 \end {array} \right) \end{small}
\]
The precision has been trimmed to three digits after the decimal point for presentation reasons. One can check that this matrix indeed has 6 negative eigenvalues that give us $2$ roots inside the domain and indeed both $2$ and $-2$ are in the interior.
\end{ex}

\begin{ex} \label{ex:order_3}
The next example is of domain of order $3$ with rotational symmetry. Such domains were first considered by Gustafsson in \cite{Gus88}. He has constructed one parameter families of quadrature domain with three nodes at $1$, $e^{2 \pi i/3}$ and $e^{4 \pi i /3}$. Some of the domains were simply connected, while others were doubly connected. We will consider here the domain obtained as the image of the upper half-plane under the map $\varphi(z) = \frac{63(z+ i)^2 (z-i)}{28 z^3 +108 i z^2 -84 z -36 i}$. By running the procedure on $p$ we obtain a $24 \times 24$ matrix with $15$ negative eigenvalues, which is again expected since $p$ has precisely three roots inside the domain and the degree of $\varphi$ is $3$. Again, we do not present the Bezout matrix associated to $p$ itself, due to its size. For presentation we compute the matrix associated to the polynomial $z$
\[
\begin{small}\left( \begin {array}{cccccc} - 1.313 & 1.393- 0.557 i&-
 0.763- 0.394 i& 0 &- 1.969- 1.137 i& 0 
\\ \noalign{\medskip} 1.393+ 0.557 i&- 1.313 & 0 &-
 0.763+ 0.394 i& 0 &- 1.969+ 1.137 i\\ \noalign{\medskip}-
 0.763+ 0.394 i& 0 &- 1.313 & 1.393+ 0.557 i&- 1.969
+ 1.137 i& 0 \\ \noalign{\medskip} 0 &- 0.763-
 0.394 i& 1.393- 0.557 i&- 1.313 & 0 &- 1.969- 1.137
 i\\ \noalign{\medskip}- 1.969+ 1.137 i& 0 &- 1.969- 1.137
 i& 0 &- 9.188 &- 10.500 \\ \noalign{\medskip}
 0 &- 1.969- 1.137 i& 0 &- 1.969+ 1.137 i&- 10.500
 &- 9.187 \end {array} \right) \end{small}
\]

\begin{figure}[H]
\caption{The disc and the cardioid with roots of $p$ in red}
\centering
\includegraphics[width=7.7cm,height=7.7cm]{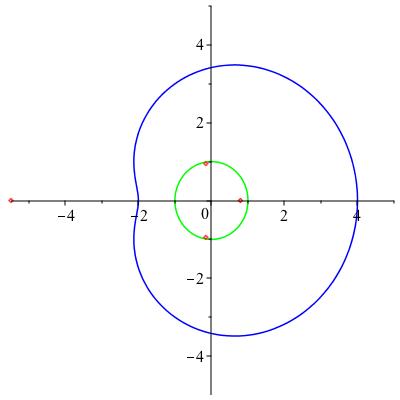}
\end{figure}
\end{ex}

\begin{figure}[H]
\caption{The Neumann oval and the domain of order $3$ with roots of $p$ in red}
\centering
\includegraphics[width=7.7cm,height=7.7cm]{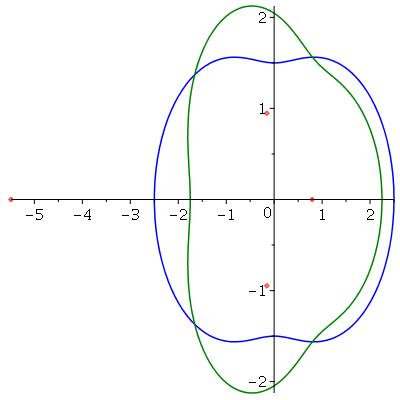}
\end{figure}
\end{ex}

\bibliographystyle{plain}
\bibliography{Quadrature_Domains}

\end{document}